\documentclass[a4paper,10pt]{amsart}
\usepackage[utf8x]{inputenc}
\usepackage{amsmath}
\usepackage{amsthm}
\usepackage{amssymb}
\usepackage{hyperref}
\usepackage{tikz}
\usepackage{amsfonts}
\usepackage[framemethod=TikZ]{mdframed}
\usetikzlibrary{positioning,arrows}
\usetikzlibrary{matrix}
\usetikzlibrary{shapes,snakes}
\usetikzlibrary{arrows,decorations.pathreplacing,patterns}

\mdfdefinestyle{Example}{%
    linecolor=blue,
    outerlinewidth=1pt,
    roundcorner=2pt,
    innertopmargin=\baselineskip,
    innerbottommargin=\baselineskip,
    innerrightmargin=20pt,
    innerleftmargin=20pt,
    backgroundcolor=gray!25!white}
    
\newtheorem{theorem}{Theorem}[section]

\newtheorem*{theorem*}{Theorem}
\newtheorem*{corollary*}{Corollary}

\definecolor{darkblue}{rgb}{0.0,0,0.7} 
\newcommand{\darkblue}{\color{darkblue}} 
\newcommand{\defn}[1]{\emph{\darkblue #1}}

\def\supp{\mathsf{supp}}

\begin{document}

\title[The number of roots of full support]{The number of roots of full support}
\author{Marko Thiel}
\address{Department of Mathematics, University of Zurich, Winterthurerstrasse 190, 8050 Zürich, Switzerland}

\begin{abstract} Chapoton has observed a simple product formula for the number of reflections in a finite Coxeter group that have full support. We give a uniform proof of his formula for Weyl groups. We furthermore refine his formula by the length of the roots.\\
\\
MSC2010: 17B22, 20F55, 52C35
\end{abstract}
\maketitle
\section{Introduction}
Let $\Phi$ be an irreducible crystallographic root system of rank $n$ with ambient space $V$, Coxeter number $h$ and exponents $e_1\leq e_2\leq\ldots\leq e_n$.
Fix a set of positive roots $\Phi^+$ for it and let $\Delta=\{\alpha_1,\alpha_2,\ldots,\alpha_n\}$ be the corresponding set of simple roots.
Let $S=\{s_{\alpha_1},s_{\alpha_2},\ldots,s_{\alpha_n}\}$ be the set of simple reflections and let $W=\langle S\rangle$ be the Weyl group of $\Phi$.
See \cite{humphreys90reflection} for background on root systems and Weyl groups.\\
\\
Any positive root $\beta\in\Phi^+$ can be written as a linear combination of simple roots 
\[\beta=\sum_{\alpha\in\Delta}c_{\beta\alpha}\alpha\]
where the coefficients $c_{\beta\alpha}$ are nonnegative integers.
Define the \defn{support} of $\beta$ as the set of those simple roots whose coefficient is nonzero:
\[\supp(\beta):=\{\alpha\in\Delta:c_{\beta\alpha}\neq0\}.\]
The aim of this note is to give a uniform proof of the following statement, observed \defn{case-by-case} by Chapoton.
\begin{theorem}[\protect{\cite[Proposition 1.1]{chapoton06sur}}]\label{main}
 The number of positive roots in $\Phi^+$ whose support is $\Delta$ is given by the formula
 \[\frac{nh}{|W|}\prod_{i=2}^n(e_i-1).\]
\end{theorem}
In fact we will prove a stronger result that refines Theorem \ref{main} by the length of the roots.
Let $\tilde{\alpha}$ be the \defn{highest root} of $\Phi$ and say that a root $\beta\in\Phi$ is \defn{long} if $\|\beta\|=\|\tilde{\alpha}\|$ and \defn{short} otherwise.
Let $n_l$ (respectively $n_s$) be the number of long (respectively short) simple roots in $\Delta$. In particular $n_l+n_s=n$. We have the following theorem.
\begin{theorem}\label{refine}
 The number of long (respectively short) positive roots in $\Phi^+$ whose support is $\Delta$ is given by the formula
 \[\frac{n_lh}{|W|}\prod_{i=2}^n(e_i-1)\quad\text{(respectively }\frac{n_sh}{|W|}\prod_{i=2}^n(e_i-1)\text{)}.\]
\end{theorem}
\section{Proof of Theorem \ref{refine}}
%
%

Define a partial order on the set of positive roots $\Phi^+$ by $\beta\leq\gamma$ if and only if $\gamma-\beta$ is a linear combination of simple roots with nonnegative coefficients.
Call $\Phi^+$ with this partial order the \defn{root poset}.
\begin{proof}[Proof of Theorem \ref{refine}]
 For a positive root $\beta\in\Phi^+$ define 
 \[I(\beta):=\{\gamma\in\Phi^+:\gamma\leq\beta\}\]
 as the principal order ideal in the root poset generated by $\beta$. We have $\supp(\beta)=I(\beta)\cap\Delta$, so the map
 \[\beta\mapsto I(\beta)\]
 is a bijection from the set of long positive roots whose support is $\Delta$ to the set of order ideals that contain $\Delta$ and whose unique maximal element is a long root.
 By a result of Sommers \cite[Proposition 6.6 (2)]{sommers05stable}, the latter set is counted by
 \[\frac{1}{[N(W_\alpha):W_\alpha]}\chi_{\alpha^\bot}(h-1),\]
 where $\alpha\in\Delta$ is some long simple root, $W_{\alpha}=\{e,s_\alpha\}$ is the parabolic subgroup generated by the reflection $s_\alpha$, $N(W_\alpha)$ is its normalizer in $W$ and $\chi_{\alpha^\bot}$ is the \defn{characteristic polynomial} of the restriction of the Coxeter arrangement to the hyperplane 
 $\alpha^\bot=\{x\in V:\langle x,\alpha\rangle=0\}$.\\
 \\
 We have that \cite[Corollary 3.10]{orlik87coxeter}
 \[\chi_{\alpha^\bot}(t)=\prod_{i=1}^{n-1}(t-e_i)\]
 and \cite[Equations (4.1) and (4.2)]{orlik83coxeter}
 \[[N(W_\alpha):W_\alpha]n_l=(-1)^{n-1}\chi_{\alpha^\bot}(-1),\]
 so 
 \begin{align*}
  \frac{1}{[N(W_\alpha):W_\alpha]}\chi_{\alpha^\bot}(h-1)&=\frac{n_l}{\prod_{i=1}^{n-1}(1+e_i)}\prod_{i=1}^{n-1}(h-1-e_i)\\
  &=\frac{n_lh}{\prod_{i=1}^{n}(1+e_i)}\prod_{i=2}^{n}(e_i-1)\\
  &=\frac{n_lh}{|W|}\prod_{i=2}^{n}(e_i-1).
 \end{align*}
 Here we used the duality of exponents $h-e_i=e_{n+1-i}$, $h=e_n+1$ and $|W|=\prod_{i=1}^{n}(1+e_i)$.
 The argument for short positive roots of full support is identical.
 \section{Acknowledgements}
 The author would like to thank the anonymous referees for pointing out references and suggesting the inclusion of the more refined Theorem \ref{refine}.
\end{proof}

\bibliographystyle{alpha}
\bibliography{literature}

\begin{thebibliography}{Hum90}

\bibitem[Cha06]{chapoton06sur}
Fr{\'e}d{\'e}ric Chapoton.
\newblock Sur le nombre de r\'eflexions pleines dans les groupes de coxeter
  finis.
\newblock {\em Bulletin of the Belgian Mathematical Society}, 13:585--596,
  2006.

\bibitem[Hum90]{humphreys90reflection}
James~E. Humphreys.
\newblock {\em {Reflection Groups and Coxeter Groups}}.
\newblock Cambridge University Press, Cambridge, 1990.

\bibitem[OS83]{orlik83coxeter}
Peter Orlik and Louis Solomon.
\newblock {Coxeter arrangements}.
\newblock {\em Proceedings of Symposia in Pure Mathematics}, 40(2):269--291,
  1983.

\bibitem[OST87]{orlik87coxeter}
Peter Orlik, Louis Solomon, and Hiroaki Terao.
\newblock {On Coxeter Arrangements and the Coxeter Number}.
\newblock {\em Advanced Studies in Pure Mathematics}, 8:461--477, 1987.

\bibitem[Som05]{sommers05stable}
Eric~N. Sommers.
\newblock {$\mathfrak{b}$-Stable Ideals in the Nilradical of a Borel
  Subalgebra}.
\newblock {\em Canadian Mathematical Bulletin}, 48:460--472, 2005.

\end{thebibliography}

\end{document}